\documentclass[11pt,reqno]{amsart}
\usepackage[utf8]{inputenc} % allow utf-8 input
\usepackage[T1]{fontenc}    % use 8-bit T1 fonts
\usepackage{hyperref}
\usepackage{amsfonts}       % blackboard math symbols
\usepackage{microtype}      % microtypography

\usepackage{amscd,amssymb,amsmath,amsthm}
\usepackage[arrow,matrix]{xy}
\usepackage{graphicx}
\usepackage{multicol}
\usepackage{subfigure}
\usepackage{epstopdf}
\usepackage{color}
\newtheorem{thm}{Theorem}
\newtheorem{defn}{Definition}
\newtheorem{lemma}{Lemma}
\newtheorem{pro}{Proposition}

\numberwithin{equation}{section} 

\begin{document}

\vspace{0.5in}
\renewcommand{\bf}{\bfseries}
\renewcommand{\sc}{\scshape}
%insert defs/styles
\vspace{0.5in}

\title[Neimark-Sacker bifurcation and stability]{Neimark-Sacker bifurcation and stability analysis in a discrete phytoplankton-zooplankton system with Holling type II functional response}

\author{S.K. Shoyimardonov}

\address{V.I.Romanovskiy Institute of Mathematics,
Tashkent, Uzbekistan.}
\email{shoyimardonov@inbox.ru}

% Current address (if needed):
%\curraddr{}

%    Information for second author (if needed):
%\author{Author Two}
%\address{}
%\email{}

%    General info
%%%%%%%%%%%%%%%%%%%%%%%%%%%%%%%%%%%%%%%%%%%%%%%%%%%

%                                                                                                                           %
%         Please use the current 2010 Mathematics Subject Classification:             %
%         http://www.ams.org/mathscinet/msc/                                                        %
%         http://www.zentralblatt-math.org/msc/en/                                                 %
%%%%%%%%%%%%%%%%%%%%%%%%%%%%%%%%%%%%%%%%%%%%%%%%%%%



\begin{abstract} In this paper, we study discrete-time model of phytoplankton-zooplankton with Holling type II predator functional response. It is shown that Neimark-Sacker bifurcation occurs at the one of positive fixed points for certain parameter chosen as a bifurcation parameter.  The existence and local stability of the positive fixed points of the model are proved. By considering theoretical results in the concrete example,   it was obtained interesting dynamics of this system, which is not investigated in its corresponding continuous system.
\end{abstract}

\maketitle

\section{Introduction}

Investigation of ocean ecosystem is important in nature and it is actual research area in the theory of dynamic systems. Marine ecosystem models can illustrate the interaction between essential organisms and elements such as phytoplankton, zooplankton, mixoplankton, carbon, bacteria etc. Various models were studied by many researchers and obtained interesting results (\cite{Hong}, \cite{Qiu},  \cite{RSH}, \cite{RSHV}, \cite{Sajib}, \cite{Chen}, \cite{Tian}). Plankton serve as the basis for the aquatic food chain and they play an important role in ocean ecosystems. Basically, the interaction between two forms of plankton, plant-plankton known as phytoplankton and animal-plankton known as zooplankton, is widely studied. Phytoplankton mainly consist of unicellular photosynthetic organisms absorbing mineral elements (nitrogen, phosphorus, calcium, iron) and transform these elements into toxin (organic matters). Phytoplankton contributes about half of the photosynthesis on the planet and absorbs one-third of the carbon dioxide.  Zooplankton feed on toxin, phytoplankton and they are key of the marine food. Therefore, it is important to study the process of interaction between phytoplankton and zooplankton.

In \cite{Chatt} the following continuous-time phytoplankton-zooplankton model is considered:

\begin{equation}\label{chat}
\begin{cases}
\frac{dP}{dt}=bP(1-\frac{P}{k})-\alpha f(P)Z,\\
\frac{dZ}{dt}=\beta f(P)Z-rZ-\theta g(P)Z,
\end{cases}
\end{equation}
where $P$ is the density of phytoplankton and $Z$ is the density of the zooplankton population; $\alpha> 0$ and $\beta> 0$
are predation and conversion rates of the zooplankton on the phytoplankton population, respectively; $b > 0$ is the growth rate, $k > 0$ is carrying capacity of the phytoplankton; $r > 0$ is the death rate of the zooplankton; $f(P)$ represents the predator functional response; $g(P)$ represents the distribution of the toxin substances; $\theta> 0$ denotes the rate of toxin liberation by the phytoplankton population. Authors of \cite{Chatt} analyzed the local stability of the model (\ref{chat}) with different kinds of $f(u)$ and $g(u)$.

In \cite{Chen}, authors investigated the model (\ref{chat}) in continuous-time by choosing $f(u)=\frac{u^h}{1+cu^h}$ (for $h = 1, 2$), $g(u) = u$ and denoting
$$\overline{t}=bt, \, \overline{u}=\frac{P}{k}, \, \overline{v}=\frac{\alpha k^{h-1}Z}{b}, \, \overline{c}={ck^h}, \, \overline{\beta}=\frac{\beta k^h}{b}, \, \overline{r}=\frac{r}{b}, \, \overline{\theta}=\frac{\theta k}{b}.$$

Then by dropping the overline sign at time $t\geq0$ we get: 

\begin{equation}\label{chenn}
\begin{cases}
\frac{du}{dt}=u(1-u)-\frac{u^hv}{1+cu^h}\\
\frac{dv}{dt}=\frac{\beta u^hv}{1+cu^h}-rv-\theta uv.\\
\end{cases}
\end{equation}

Notice that, for $h = 1, f(u)$ denotes the Holling type II predator functional response, and
for $h = 2, f(u)$ denotes the Holling type III predator functional response. In the case $\theta=0$ the global dynamics of the system (\ref{chenn})  is well studied by many mathematicians (\cite{Chen2}, \cite{Cheng2}, \cite{Hsu1}, \cite{Hsu2}, \cite{Hsu3}, \cite{Ko}, \cite{Peng}, \cite{Wang}, \cite{Zhou}). For $\theta>0$ authors (in \cite{Chen}) investigated the effect of the toxin substances and showed the occurrence of global stable and bistable phenomenons for the model (\ref{chenn}).

At time moment $t\geq 0,$ consider the model (\ref{chenn}) for $h=1$ :

\begin{equation}\label{1}
\begin{cases}
\frac{du}{dt}=u(1-u)-\frac{uv}{1+cu}\\
\frac{dv}{dt}=\frac{\beta uv}{1+cu}-rv-\theta uv,
\end{cases}
\end{equation}
where $\beta, r, \theta, c$ are positive parameters.

Let's consider discrete-time version of the model (\ref{1}), which has the following form
\begin{equation}\label{discr}
V:
\begin{cases}
u^{(1)}=u(2-u)-\frac{uv}{1+cu}\\[2mm]
v^{(1)}=\frac{\beta uv}{1+cu}+(1-r)v-\theta uv.
\end{cases}
\end{equation}
where $(u,v)\in R^2_+=\{(x,y)\in R^2:x\geq0, y\geq0\}.$

In this paper, we investigate existence and local stability of fixed points  and occurrence of Neimark-Sacker bifurcation at a positive fixed point. The paper organized as following: In the Section 2, we find conditions to parameters for existence of positive fixed points and  analyse local stability of them. In the Section 3, sufficient conditions for the occurrence of the Neimark-Sakker bifurcation are obtained. In the Section 4, we consider the concrete example with numerical simulations which illustrate our theoretical results. In the last Section we give a discussion.

\section{Fixed Points}

Recall that the fixed point  $p$ for a mapping $F: R^{m}\rightarrow R^{m}$ is a solution to the equation $F(p)=p$.
 In this section, we find conditions for parameters to be exist fixed points of the operator $\ref{discr}$ with positive coordinates  and investigate their local stability using the known lemma.   To find fixed points of the operator (\ref{discr}) we have to solve the following system:

\begin{equation}\label{fpsys}
\begin{cases}
u(2-u)-\frac{uv}{1+cu}=u\\[2mm]
\frac{\beta uv}{1+cu}+(1-r)v-\theta uv=v
\end{cases}
\end{equation}

Obviously, $E_0=(0;0)$ and $E_1=(1,0)$ are fixed points of $V.$ The case $u>0, v>0$ will be studied below (see Section 2.1).

\begin{defn}\label{def1} Let $E(x,y)$ be a fixed point of the operator $F:\mathbb{R}^{2}\rightarrow\mathbb{R}^{2}$ and $\lambda_1, \lambda_2$ are eigenvalues of the Jacobian matrix $J=J_{F}$ at the point $E(x,y).$

(i) If $|\lambda_1|<1$ and $|\lambda_2|<1$ then the fixed point $E(x,y)$ is called an \textbf{attractive} or \textbf{sink};

(ii) If $|\lambda_1|>1$ and $|\lambda_2|>1$ then the fixed point $E(x,y)$ is called  \textbf{repelling} or \textbf{source};

(iii) If $|\lambda_1|<1$ and $|\lambda_2|>1$ (or $|\lambda_1|>1$ and $|\lambda_2|<1$) then the fixed point $E(x,y)$ is called  \textbf{saddle};

(iv) If either $|\lambda_1|=1$ or $|\lambda_2|=1$ then the fixed point $E(x,y)$ is called to be  \textbf{non-hyperbolic};
\end{defn}

\begin{pro} The following statements hold true:
$$E_{0}=\left\{\begin{array}{lll}
{\rm saddle}, \ \ \ \ {\rm if} \ \  0<r<2\\[2mm]
{\rm nonhyperbolic}, \ \ {\rm if} \ \  r=2\\[2mm]
{\rm repelling}, \ \ \ \  {\rm if}  \ \ r>2,
\end{array}\right.$$
$$E_{1}=\left\{\begin{array}{lll}
{\rm attractive}, \ \ \ \ {\rm if} \ \  \frac{\beta}{1+c}<r+\theta<2+\frac{\beta}{1+c}\\[2mm]
{\rm nonhyperbolic}, \ \ {\rm if} \ \  r+\theta=\frac{\beta}{1+c} \ \ {\rm or} \ \ r+\theta=2+\frac{\beta}{1+c}  \\[2mm]
{\rm saddle}, \ \ \ \  {\rm if}  \ \  {\rm otherwise}
\end{array}\right.$$
\end{pro}

\begin{proof} The Jacobian of the operator $V$ is
\begin{equation}\label{jac}
J(u,v)=\begin{bmatrix}
2-2u-\frac{v}{(1+cu)^2} & -\frac{u}{1+cu}\\
\frac{\beta v}{(1+cu)^2}-\theta v & \frac{\beta u}{1+cu}+1-r-\theta u
\end{bmatrix}
\end{equation}
Then $J(0,0)=\begin{bmatrix}
2 & 0\\
0 & 1-r
\end{bmatrix}$  and eigenvalues are $2$ and $1-r.$ From this, for $E_0$ we can take the proof easily. Similarly, $J(1,0)=\begin{bmatrix}
0 & -\frac{1}{1+c}\\
0 & \frac{\beta}{1+c}+1-r-\theta
\end{bmatrix}$ and the eigenvalues are $\lambda_1=0, \lambda_2=\frac{\beta}{1+c}+1-r-\theta.$ By solving $|\lambda_2|<1$ we get the condition $\frac{\beta}{1+c}<r+\theta<2+\frac{\beta}{1+c}.$ Thus, the proposition is proved.
\end{proof}
\subsection{Existence of positive fixed points}
From the system (\ref{fpsys}) we  get
\begin{equation}
\begin{cases}\label{sys1}
u+\frac{v}{1+cu}=1\\[2mm]
\frac{\beta u}{1+cu}-r-\theta u=0.
\end{cases}
\end{equation}

\begin{pro}\label{prop2} The following statements hold true:

(i) If $r+\theta<\beta\leq \frac{(r+\theta)^2}{\theta}$ and $0<c<\frac{\beta-r-\theta}{r+\theta}$ then there exists unique positive fixed point $E_2=(u^*,v^*),$ (i.e., solution of (\ref{sys1})), 

(ii) If $\beta>\frac{(r+\theta)^2}{\theta}$ and $0<c<\frac{\beta-r-\theta}{r+\theta}$ then there exists unique positive fixed point $E_2=(u^*,v^*),$

(iii) If $\beta>\frac{(r+\theta)^2}{\theta}$ and $\frac{\beta-r-\theta}{r+\theta}<c<\frac{\beta+\theta-2\sqrt{\beta\theta}}{r}$ then there exist two positive fixed points $E_2=(u^*,v^*)$ and $E_3=(u^{**},v^{**});$

(iv) If $\beta>\frac{(r+\theta)^2}{\theta}$ and $c=\frac{\beta+\theta-2\sqrt{\beta\theta}}{r}$ then there exists unique positive fixed point
$E_4=(\overline{u},\overline{v}),$
where
$$u^*=\frac{\beta-rc-\theta-\sqrt{(\beta-rc-\theta)^2-4cr\theta}}{2c\theta}, \ \ v^*=(1-u^*)(1+cu^*), $$ $$u^{**}=\frac{\beta-rc-\theta+\sqrt{(\beta-rc-\theta)^2-4cr\theta}}{2c\theta}, \ \ v^{**}=(1-u^{**})(1+cu^{**}),$$  $$\overline{u}=\frac{r}{\sqrt{\theta}(\sqrt{\beta}-\sqrt{\theta})}, \ \ \overline{v}=(1-\overline{u})(1+c\overline{u}).$$
\end{pro}

\begin{proof} First, we have to solve the equation with respect to $u:$
$$c\theta u^2-(\beta-rc-\theta)u+r=0,$$
its discriminant $D=(\beta-rc-\theta)^2-4cr\theta$ is  positive iff
\begin{equation} \label{conc1}
c<\frac{\beta+\theta-2\sqrt{\beta\theta}}{r}.
 \end{equation}
 Then, the roots of (\ref{fpeq}) are

 $$u_{1}=\frac{\beta-rc-\theta-\sqrt{(\beta-rc-\theta)^2-4cr\theta}}{2c\theta}, \ \ u_{2}=\frac{\beta-rc-\theta+\sqrt{(\beta-rc-\theta)^2-4cr\theta}}{2c\theta}.$$
 Moreover, if $\beta>\theta$ then under condition (\ref{conc1}), it follows that  $\beta-rc-\theta>0.$ If we assume $\beta-rc-\theta<0,$ i.e., $\beta-\theta<rc<\beta+\theta-2\sqrt{\beta\theta},$ then $2\theta-2\sqrt{\beta\theta}>0$ which contradicts to condition $\beta>\theta.$ Since, $\sqrt{D}<\beta-rc-\theta$ it follows the positiveness of different $u_1, u_2$ with conditions $\beta>\theta$ and   $c<\frac{\beta+\theta-2\sqrt{\beta\theta}}{r}.$

 If $c=\frac{\beta+\theta-2\sqrt{\beta\theta}}{r}$ then $D=0$ and $u_1=u_2=\overline{u}=\frac{r}{\sqrt{\theta}(\sqrt{\beta}-\sqrt{\theta})}$ which is positive if $\beta>\theta.$ But, from the system (\ref{sys1}) we have $v=(1-u)(1+cu)$ and for positiveness of $v$ we have to check the condition $\overline{u}<1,$ i.e., $ \frac{r}{\sqrt{\theta}(\sqrt{\beta}-\sqrt{\theta})}<1$ which gives more stronger condition $\beta>\frac{(r+\theta)^2}{\theta}$ than $\beta>\theta.$ Hence, we proved assertion (iv) of the proposition.

  Let condition (\ref{conc1}) is satisfied and $\beta>\theta.$ Then $u_1>0$, $u_2>0$ and in the next steps we have to find conditions for positiveness of $v.$  Let us consider bigger root $u_2$ with condition $u_2<1.$ Then

  $$\frac{\beta-rc-\theta+\sqrt{(\beta-rc-\theta)^2-4cr\theta}}{2c\theta}<1 \Rightarrow \sqrt{(\beta-rc-\theta)^2-4cr\theta}<2c\theta+rc+\theta-\beta $$ \\
  If $2c\theta+rc+\theta-\beta>0$ or
  \begin{equation} \label{conc2}
  c>\frac{\beta-\theta}{r+2\theta}
\end{equation}
then from $\sqrt{(\beta-rc-\theta)^2-4cr\theta}<2c\theta+rc+\theta-\beta$ we get the  condition
  \begin{equation} \label{conc3}
  c>\frac{\beta-r-\theta}{r+\theta}.
\end{equation}
By comparing (\ref{conc2}) and (\ref{conc3}) we have that if $\beta<\frac{(r+\theta)^2}{\theta}$ then
$$\frac{\beta-\theta}{r+2\theta}>\frac{\beta-r-\theta}{r+\theta}.$$
On the other hand the condition (\ref{conc1}) must be satisfied, i.e.,
$$\frac{\beta-\theta}{r+2\theta}<\frac{\beta+\theta-2\sqrt{\beta\theta}}{r}.$$
Simplifying this inequality, we get $\sqrt{\theta}(\sqrt{\theta}-\sqrt{\beta})(\theta+r-\sqrt{\beta\theta})>0.$ Since, $\beta>\theta$ we have $\theta+r-\sqrt{\beta\theta}<0,$ i.e.,  $\beta>\frac{(r+\theta)^2}{\theta}.$ This contradiction supports that the condition $u_2<1$ can be satisfied if  $\beta>\frac{(r+\theta)^2}{\theta}$ or in the case $\frac{\beta-\theta}{r+2\theta}<\frac{\beta-r-\theta}{r+\theta}.$ Let $\beta>\frac{(r+\theta)^2}{\theta}$ and $c>\frac{\beta-r-\theta}{r+\theta}.$ If we show that both conditions (\ref{conc1}) and (\ref{conc3}) are satisfied then it follows that $u_2<1$ so $u_1<1$ and there exist two different positive fixed points.  Let us check the inequality  $$\frac{\beta-r-\theta}{r+\theta}<\frac{\beta+\theta-2\sqrt{\beta\theta}}{r}.$$
From this we get $(r+\theta-\sqrt{\beta\theta})^2>0$ which is always true except $\beta=\frac{(r+\theta)^2}{\theta}.$ If $\beta=\frac{(r+\theta)^2}{\theta}$ then $\frac{\beta-r-\theta}{r+\theta}=\frac{\beta+\theta-2\sqrt{\beta\theta}}{r}$ and from condition (\ref{conc1}) one has $c<\frac{\beta-r-\theta}{r+\theta},$ i.e., $u_2\geq1.$ Thus, we can finish the proof of assertion (iii).

In the last step we assume that
\begin{equation} \label{conc4}
  \beta>r+\theta, \ \ c<\frac{\beta-r-\theta}{r+\theta}.
\end{equation}
Obviously, in this case $u_2>1$, it is easily checked that $u_1<1.$ So, there exists unique positive fixed point $(u_1,v_1)$ which gives us the proof of assertions (i) and (ii).
Note that, if $c=\frac{\beta-r-\theta}{r+\theta}$ then $u_1=u_2=1$ and $v_1=v_2=0.$ Consequently, the proof is completed.

\end{proof}

\subsection{Stability analysis of positive fixed points}
Before analyze the fixed points we give the following useful lemma (\cite{Cheng}).
\begin{lemma}\label{lem1} Let $F(\lambda)=\lambda^2+B\lambda+C,$ where $B$ and $C$ are two real constants. Suppose $\lambda_1$ and $\lambda_2$ are two roots of $F(\lambda)=0.$ Then the following statements hold.
\begin{itemize}
 \item[] (i) If $F(1)>0$ then

(i.1) $|\lambda_1|<1$ and $|\lambda_2|<1$ if and only if $F(-1)>0$ and $C<1;$

(i.2) $\lambda_1=-1$ and $\lambda_2\neq-1$ if and only if $F(-1)=0$ and $B\neq2;$

(i.3) $|\lambda_1|<1$ and $|\lambda_2|>1$ if and only if $F(-1)<0;$

(i.4) $|\lambda_1|>1$ and $|\lambda_2|>1$ if and only if $F(-1)>0$ and $C>1;$

(i.5) $\lambda_1$ and $\lambda_2$ are a pair of conjugate complex roots and $|\lambda_1|=|\lambda_2|=1$ if and only

       if $-2<B<2$ and $C=1;$

(i.6) $\lambda_1=\lambda_2=-1$ if and only if $F(-1)=0$ and $B=2.$

\item[] (ii) If $F(1)=0,$ namely, 1 is one root of $F(\lambda)=0,$ then the other root $\lambda$ satisfies

$|\lambda|=(<,>)1$ if and only if $|C|=(<,>)1.$

\item[] (iii) If $F(1)<0,$ then $F(\lambda)=0$ has one root lying in $(1;\infty).$ Moreover,

(iii.1) the other root $\lambda$ satisfies $\lambda<(=)-1$ if and only if $F(-1)<(=)0;$

(iii.2) the other root $\lambda$ satisfies $-1<\lambda<1$ if and only if $F(-1)>0.$
\end{itemize}
\end{lemma}

\begin{pro} The fixed point $E_4=(\overline{u},\overline{v})$ mentioned in Proposition \ref{prop2} of the operator (\ref{discr}) is a non-hyperbolic fixed point.

\end{pro}
\begin{proof} Recall that for coordinates of the positive fixed points we have 
\begin{equation}\label{fpeq2}
v=(1-u)(1+cu), \ \ \ \ c\theta u^2-(\beta-rc-\theta)u+r=0
\end{equation}
and $0<u<1$ . In addition, for the fixed point $E_4,$ $\beta>\frac{(r+\theta)^2}{\theta}$ and $c=\frac{\beta+\theta-2\sqrt{\beta\theta}}{r}.$ Using (\ref{fpeq2}) if we simplify the Jacobian matrix (\ref{jac}) then we get the following form for $J(u,v):$
\begin{equation}\label{jac1}
J(u,v)=\begin{bmatrix}
(1-u)(\frac{1+2cu}{1+cu}) & -\frac{u}{1+cu}\\
(1-u)(1+cu)[\frac{\beta}{(1+cu)^2}-\theta] & 1
\end{bmatrix}
\end{equation}
The characteristic equation is
\begin{equation}\label{chareq}
F(\lambda, u)=\left((1-u)\left(\frac{1+2cu}{1+cu}\right)-\lambda\right)(1-\lambda)+u(1-u)\left(\frac{\beta}{(1+cu)^2}-\theta\right)=0.
\end{equation}
So, $F(1, u)=u(1-u)\left(\frac{\beta}{(1+cu)^2}-\theta\right).$ Let us solve the equation with respect to $\overline{u}:$
$$F(1, \overline{u})=0 \Rightarrow \overline{u}(1-\overline{u})\left(\frac{\beta}{(1+c\overline{u})^2}-\theta\right)=0,$$
since $0<\overline{u}<1,$ we get
$$\frac{\beta}{(1+c\overline{u})^2}-\theta=0 \Rightarrow \overline{u}=\frac{\sqrt{\beta}-\sqrt{\theta}}{c\sqrt{\theta}}.$$
On the other hand, $\overline{u}=\frac{r}{\sqrt{\theta}(\sqrt{\beta}-\sqrt{\theta})}.$  By equating values of $\overline{u},$ we obtain that
 $c=\frac{\beta+\theta-2\sqrt{\beta\theta}}{r}$ which is necessary condition for existence the positive fixed point $E_4.$ Thus, $F(1, \overline{u})=0,$ i.e., one eigenvalue equals to 1, by Definition \ref{def1} the fixed point $E_4$ is non-hyperbolic.
\end{proof}
Assume that the characteristic equation (\ref{chareq}) has the form $F(\lambda, u)=\lambda^2-p(u)\lambda+q(u)=0,$ where
\begin{equation}\label{bif1}
p(u)=(1-u)\left(\frac{1+2cu}{1+cu}\right)+1, \ \ \ \ q(u)=(1-u)\left(\frac{1+2cu}{1+cu}\right)+u(1-u)\left(\frac{\beta}{(1+cu)^2}-\theta\right)
\end{equation}

\begin{lemma}\label{lem2} For the fixed point $E_2=(u^*,v^*)$ of the operator (\ref{discr}), the followings hold true
$$E_{2}=\left\{\begin{array}{lll}
{\rm attractive}, \ \ {\rm if} \ \  q(u^*)<1\\[2mm]
{\rm repelling}, \ \ \ \ {\rm if} \ \  q(u^*)>1\\[2mm]
{\rm nonhyperbolic}, \ \ \ \  {\rm if}  \ \ p(u^*)<2, \ \ q(u^*)=1,
\end{array}\right.$$
where, $u^*=\frac{\beta-rc-\theta-\sqrt{(\beta-rc-\theta)^2-4cr\theta}}{2c\theta}.$
\end{lemma}

\begin{proof}\textbf{ Step-1}. By equation (\ref{chareq}), let's check the sign of $F(1,u^*):$
$$F(1,u^*)=u^*(1-u^*)\left(\frac{\beta}{(1+cu^*)^2}-\theta\right)>0 \ \ \Leftrightarrow \ \ \frac{\beta}{(1+cu^*)^2}-\theta>0,$$

$$\Leftrightarrow \ \ u^*<\frac{\sqrt{\beta}-\sqrt{\theta}}{c\sqrt{\theta}} \ \ \Leftrightarrow \ \ \frac{\beta-rc-\theta-\sqrt{(\beta-rc-\theta)^2-4cr\theta}}{2c\theta}<\frac{\sqrt{\beta}-\sqrt{\theta}}{c\sqrt{\theta}}$$
$$\Rightarrow \beta-rc-\theta-\sqrt{(\beta-rc-\theta)^2-4cr\theta}<2\sqrt{\beta\theta}-2\theta \Rightarrow (\sqrt{\beta}-\sqrt{\theta})^2<rc+\sqrt{(\beta-rc-\theta)^2-4cr\theta}$$
Since, $(\sqrt{\beta}-\sqrt{\theta})^2-rc>0$ we have
$$(\beta-rc-\theta)^2-4cr\theta>(\beta+\theta-rc-2\sqrt{\beta\theta})^2 \Rightarrow \sqrt{\beta\theta}(\beta+\theta-rc)>2\beta\theta \Rightarrow $$
$\Rightarrow c<\frac{(\sqrt{\beta}-\sqrt{\theta})^2}{r}.$
Recall that, the last inequality is necessary condition to existence of positive fixed point in Lemma \ref{lem2}. Hence, $F(1,u^*)>0$ is always true. \\
\textbf{Step-2}. In this step, we study the sign of $F(-1,u^*).$
$$F(-1,u^*)=2\left((1-u^*)\left(\frac{1+2cu^*}{1+cu^*}\right)+1\right)+u^*(1-u^*)\left(\frac{\beta}{(1+cu^*)^2}-\theta\right).$$
In the first step, we have shown that $\frac{\beta}{(1+cu^*)^2}-\theta>0$ is always true. Thus, $F(-1,u^*)>0$ also always true.\\
\textbf{Step-3}. In the previous steps we have shown that for the fixed point $E_2,$ $F(1,u^*)>0$ and $F(-1,u^*)>0,$ by assertions (i.1), (i.4) of Lemma \ref{lem1} we get the proof of first two assertions of the theorem. By assertion (i.5), if $p(u^*)<2$ and $q(u^*)=1$ then the characteristic equation (\ref{chareq}) has the pair of complex conjugate eigenvalues with module 1. So, we can complete the proof of the lemma. Note that the parameters can be chosen such that each case in the lemma holds.
\end{proof}

\begin{pro} For the fixed point $E_3=(u^{**},v^{**})$ of the operator (\ref{discr}), the followings hold true
$$E_{3}=\left\{\begin{array}{lll}
{\rm saddle}, \ \ {\rm if} \ \  F(-1, u^{**})>0\\[2mm]
{\rm repelling}, \ \ \ \ {\rm if} \ \  F(-1, u^{**})<0\\[2mm]
{\rm nonhyperbolic}, \ \ \ \  {\rm if}  \ \ F(-1, u^{**})=0,
\end{array}\right.$$
\end{pro}

\begin{proof} We consider the inequality $F(1,u^{**})<0:$
$$F(1,u^{**})=u^{**}(1-u^{**})\left(\frac{\beta}{(1+cu^{**})^2}-\theta\right)<0 \ \ \Leftrightarrow \ \ \frac{\beta}{(1+cu^{**})^2}-\theta<0,$$

$$\Leftrightarrow \ \ u^{**}>\frac{\sqrt{\beta}-\sqrt{\theta}}{c\sqrt{\theta}} \ \ \Leftrightarrow \ \ \frac{\beta-rc-\theta+\sqrt{(\beta-rc-\theta)^2-4cr\theta}}{2c\theta}>\frac{\sqrt{\beta}-\sqrt{\theta}}{c\sqrt{\theta}}$$
$\Rightarrow \beta-rc-\theta+\sqrt{(\beta-rc-\theta)^2-4cr\theta}>2\sqrt{\beta\theta}-2\theta \Rightarrow$ $$\sqrt{(\beta-rc-\theta)^2-4cr\theta}>rc-(\sqrt{\beta}-\sqrt{\theta})^2.$$
Last inequality is always true, because, $(\sqrt{\beta}-\sqrt{\theta})^2-rc>0.$  Thus, $F(1,u^{**})<0$ is always true and by Lemma \ref{lem1} , one eigenvalue belongs to $(1;\infty).$ All three conditions of the proposition follow directly from (iii.1) and (iii.2) of the Lemma \ref{lem1} . The proof is completed.

\end{proof}

\section{Neimark-Sacker bifurcation analysis}

In this section we obtain conditions for occurrence of Neimark-Sacker bifurcation at the fixed point $E_2(u^*,v^*)$. First, we give the following definitions and well-known theorems.

Recall that in the dynamical system $(T,X,\phi^t),$ $T$ is a time set, $X$ is a state space and  $\phi^t: X\rightarrow X$ is a family of evolution operators parameterized by $t\in T.$ 

\begin{defn}(see \cite{Kuz}) A dynamical system $\{T, \mathbb{R}^n, \varphi^t\}$ is called \textbf{locally topologically equivalent} near a fixed point $x_0$ to a dynamical system $\{T, \mathbb{R}^n, \psi^t\}$
near a fixed point  $y_0$ if there exists a homeomorphism $h : \mathbb{R}^n \rightarrow \mathbb{ R}^n$ that
is

(i) defined in a small neighborhood $U\subset \mathbb{R}^n$ of $x_0$;

(ii) satisfies $y_0 = h(x_0)$;

(iii) maps orbits of the first system in $U$ onto orbits of the second system
in $V = f(U)\subset\mathbb{R}^n $, preserving the direction of time.
\end{defn}

Recall that, the phase portrait of a dynamical system is a partitioning of the state space into orbits. In the dynamical system depending on parameters, if parameters vary then the phase portrait also varies. There are two possibilities: either the system remains topologically equivalent to the original one, or its topology changes.

\begin{defn}(see \cite{Kuz}) The appearance of a topologically nonequivalent phase portrait under variation of parameters is called a \textbf{bifurcation}.
\end{defn}

Suppose that given two-dimensional discrete-time system depending on parameters and its Jacobian matrix at the nonhyperbolic fixed point has two complex conjugate eigenvalues $\mu_{1,2}$ with modules one.

\begin{defn} (see \cite{Kuz})  The bifurcation corresponding to the presence of $\mu_{1,2}$ is called a \textbf{Neimark-Sacker} (or torus) bifurcation.

\end{defn}

From the third case of the Lemma \ref{lem2}, we obtain that at the positive fixed point $E_2(u^*,v^*)$ the Jacobian has a pair of complex conjugate eigenvalues with modules 1 if $p(u^*)<2$ and $q(u^*)=1,$ where  $p(u^*), q(u^*)$ are defined as (\ref{bif1}).

We notice that all parameters belong to the set:
$$S_{E_2}=\left\{(r,c,\beta,\theta)\in(0,+\infty): \ \ c<\frac{\beta+\theta-2\sqrt{\beta\theta}}{r}, \ \ \theta=\theta_0  \right\}$$
and assume that $p(u^*)<2, \ \ q(u^*)=1$ in the set $S_{E_2}.$

Using Wolfram Alpha we obtained that $q(u^*)<1$ (i.e., $E_2$ is an attractive) if $\theta>\theta_0$ and $q(u^*)>1$  (i.e., $E_2$ is repelling) if $\theta<\theta_0.$

The fixed point $E_2(u^*,v^*)$ can pass through a Neimark-Sacker bifurcation when
the parameters $(r,c,\beta,\theta)\in S_{E_2}$ and $\theta$ varies in the small neighborhood of $\theta_0$.

We choose the parameter $\theta$ as a bifurcation parameter to study the Neimark-Sacker bifurcation for the positive fixed
point $E_2(u^*,v^*)$ of the system (\ref{discr}) by using the Center Manifold Theorem and
bifurcation theory (see \cite{Guc}, \cite{Kuz}, \cite{Rob}, \cite{Wing}).

Let's consider the system (\ref{discr}) with parameters $(r,c,\beta,\theta)\in S_{E_2}$, which is described by

\begin{equation}\label{bif2}
\begin{cases}
u\rightarrow u(2-u)-\frac{uv}{1+cu}\\[2mm]
v\rightarrow \frac{\beta uv}{1+cu}+(1-r)v-\theta_0 uv.
\end{cases}
\end{equation}

\textbf{The first step}. Giving a perturbation $\theta_*$ of parameter $\theta_0,$ we consider a perturbation of the system (\ref{bif2}) as follows:

 \begin{equation}\label{bif3}
\begin{cases}
u\rightarrow u(2-u)-\frac{uv}{1+cu}\\[2mm]
v\rightarrow \frac{\beta uv}{1+cu}+(1-r)v-(\theta_0+\theta_*) uv.
\end{cases}
\end{equation}
where $|\theta_*|\ll1.$

\textbf{The second step}. Let $x=u-u^*$ and $y=v-v^*,$ which transform the fixed point $E_2(u^*,v^*)$ to the origin (0,0) and system (\ref{bif3}) into

 \begin{equation}\label{bif4}
\begin{cases}
x\rightarrow (x+u^*)(2-u^*-x)-\frac{(x+u^*)(y+v^*)}{1+cu^*+cx}-u^*\\[2mm]
y\rightarrow (y+v^*)\left(\frac{\beta (x+u^*)}{1+cu^*+cx}+1-r-(\theta_0+\theta_*)(x+u^*)\right)-v^*.
\end{cases}
\end{equation}
The Jacobian of the system (\ref{bif4}) at the point (0,0) is

\begin{equation}\label{jac2}
J(0,0)=\begin{bmatrix}
(1-u^*)(\frac{1+2cu^*}{1+cu^*}) & -\frac{u^*}{1+cu^*}\\
(1-u^*)(1+cu^*)[\frac{\beta}{(1+cu^*)^2}-\theta_0-\theta_*] & 1-\theta_*u^*
\end{bmatrix}
\end{equation}
and its characteristic equation is
$$\lambda^2-a(\theta_*)\lambda+b(\theta_*)=0,$$
where
$$a(\theta_*)=Tr(J)=\frac{(1-u^*)(1+2cu^*)}{1+cu^*}+1-\theta_*u^*,$$
 and

 $b(\theta_*)=det(J)=
(1-\theta_*u^*)\frac{(1-u^*)(1+2cu^*)}{1+cu^*}+u^*(1-u^*)[\frac{\beta}{(1+cu^*)^2}-\theta_0-\theta_*]=$
$$=1-\frac{\theta_*u^*(1-u^*)(2+3cu^*)}{1+cu^*}.$$

The roots are
\begin{equation}\label{bif5}
\lambda_{1,2}=\frac{1}{2}[a(\theta_*)\pm i \sqrt{4b(\theta_*)-a^2(\theta_*)}].
\end{equation}
Thus,
\begin{equation}\label{bif6}
|\lambda_{1,2}|=\sqrt{b(\theta_*)}
\end{equation}
and

\begin{equation}\label{bif7}
\frac{d|\lambda_{1,2}|}{d\theta_*}\Bigm|_{\theta_*=0}=-\frac{1}{2\sqrt{b(\theta_*)}}\frac{u^*(1-u^*)(2+3cu^*)}{1+cu^*}\Bigm|_{\theta_*=0}=-\frac{u^*(1-u^*)(2+3cu^*)}{2(1+cu^*)}<0.
\end{equation}
Hence, the transversality condition is satisfied. In addition, it is required the nondegeneracy condition (no strong resonance) $\lambda_{1,2}^i\neq1, \ \ i=1,2,3,4$ when $\theta_*=0.$ Since $1<a(0)=\frac{(1-u^*)(1+2cu^*)}{1+cu^*}+1<2$ and $b(0)=1$ it can be shown that
\begin{equation}\label{nondeg}
 \lambda_{1,2}^m(0)\neq1, \ \   m=1,2,3,4.
\end{equation}
\textbf{The third step}. In order to derive the normal form of the system (\ref{bif4}) when $\theta_*=0$,
we expand the system (\ref{bif4}) as Taylor series at $(x,y)=(0,0)$ up to the following third-order

\begin{equation}\label{bif8}
\begin{cases}
x\rightarrow a_{10}x+a_{01}y+a_{20}x^2+a_{11}xy+a_{02}y^2+a_{30}x^3+a_{21}x^2y+a_{12}xy^2+a_{03}y^3+O(\rho_1^4)\\[2mm]
y\rightarrow b_{10}x+b_{01}y+b_{20}x^2+b_{11}xy+b_{02}y^2+b_{30}x^3+b_{21}x^2y+b_{12}xy^2+b_{03}y^3+O(\rho_1^4),
\end{cases}
\end{equation}
where $\rho_1=\sqrt{x^2+y^2},$
\begin{equation}
\begin{split}
&a_{10}=\frac{(1-u^*)(1+2cu^*)}{1+cu^*}, \ \ a_{01}=-\frac{u^*}{1+cu^*}, \ \ a_{20}=\frac{c(1-u^*)}{(1+cu^*)^2}-1,\\
&a_{11}=-\frac{1}{(1+cu^*)^2}, \ \ a_{02}=a_{03}=a_{12}=0, \\
&a_{30}=-\frac{c^2(1-u^*)}{(1+cu^*)^3}, \ \ a_{21}=\frac{c}{(1+cu^*)^3},\\
&b_{10}=(1-u^*)(1+cu^*)\left(\frac{\beta}{(1+cu^*)^2}-\theta_0\right), \ \ b_{01}=1,\\
&b_{02}=b_{03}=b_{12}=0, \ \ b_{20}=\frac{\beta c(1-u^*)}{(1+cu^*)^2}, \ \ b_{11}=\frac{\beta}{(1+cu^*)^2},\\
&b_{21}=-\frac{\beta c}{(1+cu^*)^3}, \ \ b_{30}=\frac{\beta c^2(1-u^*)}{(1+cu^*)^3}.
\end{split}
\end{equation}
Then
$$J(E_2)=\begin{bmatrix}
a_{10} & -a_{01}\\
b_{10} & b_{01}
\end{bmatrix}  \ \ \Rightarrow \ \ J(E_2)=\begin{bmatrix}
K & -\frac{u^*}{1+cu^*}\\
m & 1
\end{bmatrix}$$
where $K=\frac{(1-u^*)(1+2cu^*)}{1+cu^*}$ and $m=(1-u^*)(1+cu^*)\left(\frac{\beta}{(1+cu^*)^2}-\theta_0\right).$ Two eigenvalues of the matrix $J(E_2)$ are
$$\lambda_{1,2}=\frac{1+K\pm i\sqrt{-D}}{2},$$
where   $D=(1+K)^2-4<0$,   since  $1<1+K<2.$  Let us find eigenvectors corresponding to $\lambda_{1,2}.$ For eigenvalue $\lambda_1=\frac{1+K+i\sqrt{-D}}{2},$ the matrix equation is

$$(J-\lambda_1 I_2)\overline{v}_1=\begin{bmatrix}
\frac{K-1-i\sqrt{-D}}{2} & -\frac{u^*}{1+cu^*}\\
m & \frac{1-K-i\sqrt{-D}}{2}
\end{bmatrix}\begin{bmatrix}x_1\\ y_1\end{bmatrix}=\begin{bmatrix}0\\0\end{bmatrix}.$$
If we multiply first row by $-\frac{2m}{K-1-i\sqrt{-D}}$ and add to second row then we get the following equation for existence nonzero eigenvector:
$$u^*(1-u^*)\left(\frac{\beta}{(1+cu^*)^2}-\theta_0\right)+\frac{(1-u^*)(1+2cu^*)}{1+cu^*}-1=0$$
which is always true from $q(u^*)=1.$ Thus, first eigenvector is

$$v_1=\begin{bmatrix}2u^*\\(K-1)(1+cu^*)\end{bmatrix}-i\begin{bmatrix}0\\ \sqrt{-D}(1+cu^*)\end{bmatrix}.$$
Similarly, it is easy to find that next eigenvector is
$$v_2=\begin{bmatrix}2u^*\\(K-1)(1+cu^*)\end{bmatrix}+i\begin{bmatrix}0\\ \sqrt{-D}(1+cu^*)\end{bmatrix}.$$

\textbf{The fourth step}. We find the normal form of the system (\ref{bif4}). Let matrix

$T= \begin{bmatrix} 0 & 2 u^*\\ \sqrt{-D}(1+cu^*) & (K-1)(1+cu^*)\end{bmatrix}$ \ \ then \ \
 $T^{-1}= \begin{bmatrix} \frac{1-K}{2u^*\sqrt{-D}} & \frac{1}{\sqrt{-D}(1+cu^*)} \\ \frac{1}{2u^*} & 0\end{bmatrix}.$

By transformation, we get that
\begin{equation}\label{bif10}
(x,y)^T=T(X,Y)^T
\end{equation}
the system (\ref{bif8}) transforms into the following system

\begin{equation}\label{bif9}
\begin{cases}
X\rightarrow \frac{K+1}{2}X+\frac{(1-K)(K+3)}{2\sqrt{-D}}Y+F(X,Y)+O(\rho_2^4)\\[2mm]
Y\rightarrow -\frac{\sqrt{-D}}{2}X+\frac{K+1}{2}Y+G(X,Y)+O(\rho_2^4).
\end{cases}
\end{equation}
where $\rho_2^4=\sqrt{X^2+Y^2}$ and
\begin{equation}
\begin{split}
&F(X,Y)=c_{02}Y^2+c_{03}Y^3+c_{11}XY+c_{12}XY^2,\\
&G(X,Y)=d_{02}Y^2+d_{03}Y^3+d_{11}XY+d_{12}XY^2,\\
\end{split}
\end{equation}
For simplicity, denote  $s=1-c+2cu^*$ then $(1-K)(1+cu^*)=u^*s$ and from (\ref{bif8}), (\ref{bif10}) we obtain
\begin{equation}
\begin{split}
&c_{02}=\frac{(u^{*})^2}{\sqrt{-D}(1+cu^*)}(4b_{20}+s(2a_{20}-2b_{11}-a_{11}s)),\\
&c_{03}=\frac{2(u^*)^3}{\sqrt{-D}(1+cu^*)}(4b_{30}+s(2a_{30}-2b_{21}-a_{21}s)),\\
&c_{11}=\frac{u^*(2\beta-s)}{(1+cu^*)^2},\ \ c_{12}=-\frac{2c(u^*)^2(2\beta-s)}{(1+cu^*)^3},\\
&d_{02}=\frac{2u^*(c(1-u^*)-(1+cu^*)^2)-u^*s}{(1+cu^*)^2},\\
&d_{03}=\frac{2c(u^*)^2(c(1-u^*)-s)}{(1+cu^*)^3},\\
&d_{11}=-\frac{\sqrt{-D}}{1+cu^*}, \ \ d_{12}=\frac{2cu^*\sqrt{-D}}{(1+cu^*)^2}.
\end{split}
\end{equation}

In addition, the partial derivatives at (0,0) are

\begin{equation}
\begin{split}
&F_{XX}=F_{XXX}=F_{XXY}=0, \ \ F_{XY}=\frac{u^*(2\beta-s)}{(1+cu^*)^2}, \\
&F_{XYY}=-\frac{4c(u^*)^2(2\beta-s)}{(1+cu^*)^3},\\
&F_{YY}=\frac{2(u^{*})^2}{\sqrt{-D}(1+cu^*)}(4b_{20}+s(2a_{20}-2b_{11}-a_{11}s)), \\
&F_{YYY}=\frac{12(u^*)^3}{\sqrt{-D}(1+cu^*)}(4b_{30}+s(2a_{30}-2b_{21}-a_{21}s)),\\
&G_{XX}=G_{XXX}=G_{XXY}=0, \ \ G_{XY}=-\frac{\sqrt{-D}}{1+cu^*}, \\
&G_{XYY}=-\frac{4cu^*\sqrt{-D}}{(1+cu^*)^2},\\
&G_{YY}=\frac{2(2u^*(c(1-u^*)-(1+cu^*)^2)-u^*s)}{(1+cu^*)^2}, \\
&G_{YYY}=\frac{12c(u^*)^2(c(1-u^*)-s)}{(1+cu^*)^3}.\\
\end{split}
\end{equation}

\textbf{The fifth step}. We need to compute the discriminating quantity $L$ via the following formula (see \cite{Rob}), which determines
the stability of the invariant circle bifurcated from Neimark-Sacker bifurcation of the system (\ref{bif9}):
\begin{equation}\label{lya}
L=-Re\left[\frac{(1-2\lambda_1)\lambda_2^2}{1-\lambda_1}L_{11}L_{20}\right]-\frac{1}{2}|L_{11}|^2-|L_{02}|^2+Re(\lambda_2 L_{21}),
\end{equation}
where
\begin{equation}
\begin{split}
&L_{20}=\frac{1}{8}[(F_{XX}-F_{YY}+2G_{XY})+i(G_{XX}-G_{YY}-2F_{XY})],\\
&L_{11}=\frac{1}{4}[(F_{XX}+F_{YY})+i(G_{XX}+G_{YY})],\\
&L_{02}=\frac{1}{8}[(F_{XX}-F_{YY}-2G_{XY})+i(G_{XX}-G_{YY}+2F_{XY})],\\
&L_{21}=\frac{1}{16}[(F_{XXX}+F_{XYY}+G_{XXY}+G_{YYY})+i(G_{XXX}+G_{XYY}-F_{XXY}-F_{YYY})].
\end{split}
\end{equation}

After some computation we get
\begin{equation}
\begin{split}
&L_{20}=\frac{1}{4}\left(\frac{D-(u^*)^2(4b_{20}+s(2a_{20}-2b_{11}-a_{11}s))}{\sqrt{-D}(1+cu^*)}\right)-\\
&   \hspace{1cm} -\frac{i}{2}\left(\frac{u^*(c(1-u^*)-(1+cu^*)^2+\beta-s)}{(1+cu^*)^2}\right), \\
&L_{11}=\frac{1}{2}\left( \frac{(u^*)^2(4b_{20}+s(2a_{20}-2b_{11}-a_{11}s))}{\sqrt{-D}(1+cu^*)}\right)+\\
&  \hspace{1cm} +\frac{i}{2}\left(\frac{2u^*(c(1-u^*)-(1+cu^*)^2)-u^*s}{(1+cu^*)^2}\right),\\
&L_{02}=-\frac{1}{4}\left(\frac{D+(u^*)^2(4b_{20}+s(2a_{20}-2b_{11}-a_{11}s))}{\sqrt{-D}(1+cu^*)}\right)-\\
&  \hspace{1cm} -\frac{i}{2}\left(\frac{u^*(c(1-u^*)-(1+cu^*)^2-\beta)}{(1+cu^*)^2}\right),\\
&L_{21}=\frac{1}{4}\left(\frac{c(u^*)^2(3c-3cu^*-2s-2\beta)}{(1+cu^*)^3}\right)-\\
&  \hspace{1cm} -\frac{i}{4}\left(\frac{cDu^*-3(u^*)^3(1+cu^*)(4b_{30}+s(2a_{30}-2b_{21}-a_{21}s))}{\sqrt{-D}(1+cu^*)^2}\right).
\end{split}
\end{equation}

Thus, from (\ref{bif7}) and (\ref{nondeg}) it is clear that the transversality condition and the nondegeneracy condition of the system (\ref{discr}) are satisfied. So, summarizing the above discussions, we obtain the following concluding theorem.

\begin{thm}\label{bifurcation} Assume the parameters $r,c,\beta,\theta$ in the set
$$S_{E_2}=\left\{(r,c,\beta,\theta)\in(0,+\infty): \ \  c<\frac{\beta+\theta-2\sqrt{\beta\theta}}{r}, \ \ \theta=\theta_0  \right\}$$
and  $L$ be defined as  (\ref{lya}).
If $L\neq0$ then the system (\ref{discr}) undergoes a Neimark-Sacker bifurcation at the fixed point $E_2(u^*,v^*)$ when the parameter $\theta_*$ varies in the small neighborhood of origin. Moreover, if $L< 0$ (resp., $L>0$), then an attracting
(resp., repelling) invariant closed curve bifurcates from the fixed point for $\theta_*>0$ (resp., $\theta_*<0$).
\end{thm}

\section{Numerical simulations}

The following example illustrates the above Theorem \ref{bifurcation}:

\textbf{Example 1}.  Let us consider the system (\ref{discr}) with parameters $c=1, \beta=4, r=\frac{10}{9}, \theta=\theta_0=\frac{4}{9}.$ Then the fixed point $E_2=(0.5,0.75)$ with the multipliers $\lambda_1=\frac{5-i\sqrt{11}}{6}$ and  $\lambda_2=\frac{5+i\sqrt{11}}{6}.$ Moreover, $|\lambda_{1,2}|=1, \frac{d|\lambda_{1,2}|}{d\theta_*}\Bigm|_{\theta_*=0}=-\frac{7}{24}<0$ and
$$L_{20}=-\frac{17}{36\sqrt{11}}-\frac{5}{36}i, \ \ L_{11}=-\frac{5}{18\sqrt{11}}-\frac{i}{2}, $$
$$L_{02}=\frac{27}{36\sqrt{11}}+\frac{23}{36}i, \ \ L_{21}=-\frac{17}{108}+\frac{159}{162\sqrt{11}}i,$$
and
$$L\approx-0,8286<0.$$
Thus, according to Theorem \ref{bifurcation}, an attracting invariant closed curve bifurcates
from the fixed point for $\theta_*>0.$

For this example, Figures \ref{outside} (a)-d) show that the closed curve is stable outside, while Figures \ref{inside}
 (a)-(d) indicate that the closed curve is stable inside for the repelling fixed point
$E_2$ as long as the assumptions of Theorem \ref{bifurcation} hold.

\begin{figure}[h!]
    \centering
    \subfigure[\tiny$\theta=0.45, n=800$]{\includegraphics[width=0.4\textwidth]{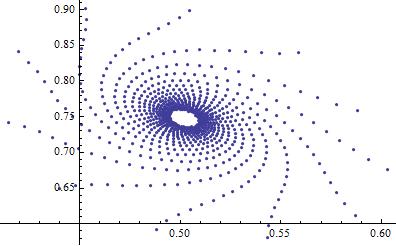}}
    \subfigure[\tiny$\theta=0.45, n=10000$]{\includegraphics[width=0.4\textwidth]{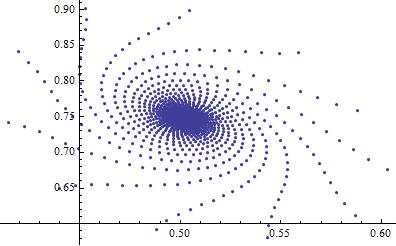}}
    \subfigure[\tiny$\theta=0.444, n=10000$]{\includegraphics[width=0.4\textwidth]{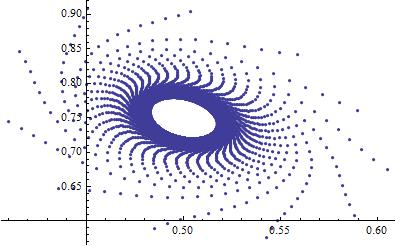}}
    \subfigure[\tiny$\theta=0.44, n=10000$]{\includegraphics[width=0.4\textwidth]{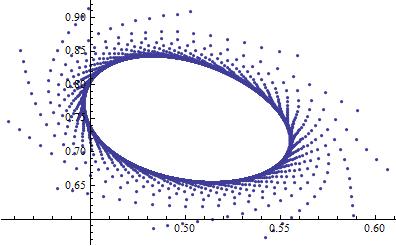}}
    \caption{Phase portraits for the system (\ref{discr}) with $c=1, \beta=4, r=10/9, \theta_0=4/9\approx0.4444..., (u^0,v^0)=(0.6, 0.75).$}
    \label{outside}
\end{figure}

\begin{figure}[h!]
    \centering
    \subfigure[\tiny $\theta=0.44, (u^0,v^0)=(0.48, 0.74), n=1000$]{\includegraphics[width=0.38\textwidth]{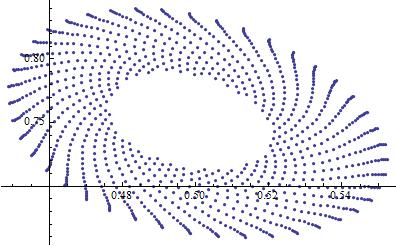}}
    \subfigure[\tiny $\theta=0.44, (u^0,v^0)=(0.48, 0.74), n=10000$]{\includegraphics[width=0.38\textwidth]{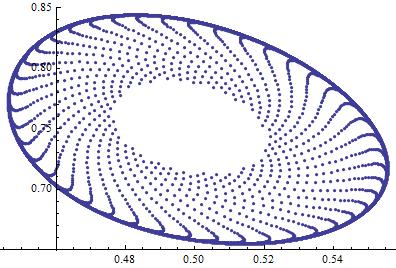}}
    \subfigure[\tiny $\theta=0.44, (u^0,v^0)=(0.45, 0.76), n=10000$]{\includegraphics[width=0.38\textwidth]{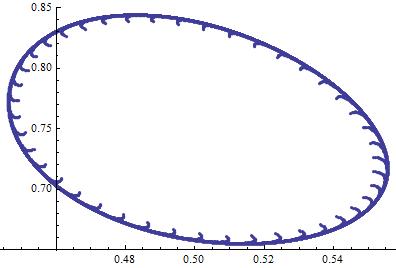}}
    \subfigure[\tiny $\theta=0.4, (u^0,v^0)=(0.48, 0.74), n=10000$]{\includegraphics[width=0.38\textwidth]{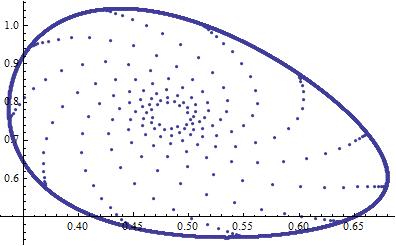}}
    \caption{Phase portraits for the system (\ref{discr}) with $c=1, \beta=4, r=10/9, \theta_0=4/9\approx0.4444....$}
    \label{inside}
\end{figure}

In addition, in the figures (a) and  (b) of the Figure \ref{outside}, the fixed point $E_2$ is an attractive fixed point because $\theta>\theta_0$ and for the other figures $E_2$ is a repelling fixed point.

\section{Discussion}

In this paper, we investigated the phytoplankton-zooplankton discrete-time model with Holling type II predator functional response. We defined type of fixed points  $E_{0}=(0,0),$ $E_{1}=(1,0)$ and found conditions for parameters that positive fixed points  $E_{2}=(u^*,v^*),$ $E_{3}=(u^{**},v^{**})$ and $E_{4}=(\overline{u},\overline{v})$ exist, here the sufficient conditions are $\beta>r+\theta$ and $cr\leq(\sqrt{\beta}-\sqrt{\theta})^2$ . In addition, we studied local stability of the fixed points $E_{2}, E_{3}, $ and $E_{4}.$  Moreover, by choosing bifurcation parameter $\theta,$ we obtained the sufficient conditions for Neimark-Sacker bifurcation to occur. By $\theta_0$ we denoted the value of $\theta$ which for $q(u^*)=1.$ Then by Lemma \ref{lem2}, $E_{2}$ is an attractive if $q(u^*)<1$ and repelling when $q(u^*)>1.$ Thus, it has been shown to be a Neimark-Sacker bifurcation is that the system (\ref{discr}) undergoes a bifurcation when the parameter $\theta$ passes through the value $\theta_0.$ Finally, we have given an example with numerical simulation illustrating our results and an attracting invariant closed curve bifurcates from the fixed point $E_{2}.$ One aspect of our future work is focused to study the global dynamics of the nonlinear model (\ref{discr}).


\end{document}